\documentclass[a4paper,12pt]{amsart}
\usepackage{graphicx}
\usepackage{amsmath,amssymb,amsthm,amsfonts}
\usepackage{amssymb}
\usepackage[active]{srcltx}

\newtheorem{thm}{Theorem}[section]

\newtheorem{lem}[thm]{Lemma}

\newtheorem{rem}[thm]{Remark}
\newtheorem{defn}[thm]{Definition}

\newcommand{\bremark}{\begin{rem} \textup}
\newcommand{\eremark}{\end{rem} }

\newcommand{\cuad}{{\sqcap\kern-.68em\sqcup}}

\newcommand{\R}{{\mathbb{R}}}

\renewcommand{\rho}{\varrho}
\renewcommand{\theta}{\vartheta}

\begin{document}

\parindent 0pc
\parskip 6pt
\overfullrule=0pt

\title{Monotonicity and symmetry of singular solutions to quasilinear problems}

\author{Francesco Esposito$^{*}$}
\address{$^{**}$Dipartimento di Matematica e Informatica, UNICAL\\
Ponte Pietro  Bucci 31B \\
87036 Arcavacata di Rende, Cosenza, Italy}
\email{ esposito@mat.unical.it}

\author{Luigi Montoro$^{*}$}
\address{$^{**}$Dipartimento di Matematica e Informatica, UNICAL\\
Ponte Pietro  Bucci 31B \\
87036 Arcavacata di Rende, Cosenza, Italy}
\email{montoro@mat.unical.it}

\author{Berardino Sciunzi$^{*}$}
\address{$^{**}$Dipartimento di Matematica e Informatica, UNICAL\\
Ponte Pietro  Bucci 31B \\
87036 Arcavacata di Rende, Cosenza, Italy}
\email{ sciunzi@mat.unical.it}


\maketitle

\begin{abstract}
We consider singular solutions to quasilinear elliptic equations under zero Dirichlet boundary condition. Under suitable assumptions on the nonlinearity we deduce symmetry and monotonicity properties of positive solutions  via an improved moving plane procedure.
\end{abstract}

\section{introduction}
We consider the problem
\begin{equation}\tag{$\mathcal P_\Gamma$}\label{problem}
\begin{cases}
-\Delta_p u=f(u) & \text{in}\,\,\Omega\setminus \Gamma  \\
u> 0 &  \text{in}\,\,\Omega\setminus \Gamma  \\
u=0 &  \text{on}\,\,\partial \Omega,
\end{cases}
\end{equation}
in a bounded smooth domain $\Omega\subset \mathbb{R}^n$ and $p>1$. The solution $u$ has a possible singularity on the critical set $\Gamma$ and in fact we shall only assume that $u$ is of class $C^{1}$ far from the critical set. Therefore the equation is understood as in the following
\begin{defn}\label{def:sol} We say that
$u\in C^{1}({\overline\Omega}\setminus \Gamma)$ is a solution to \eqref{problem} if $u=0$ on $\partial \Omega$ and
\begin{equation}\label{debil1}
\int_\Omega |\nabla u|^{p-2} (\nabla u, \nabla
\varphi)\,dx\,=\,\int_\Omega f(u)\varphi\,dx\qquad\forall \varphi\in
C^1_c(\Omega\setminus\Gamma)\,.
\end{equation}
\end{defn}
The purpose of the paper is to investigate symmetry and monotonicity properties of the solutions when the domain is assumed to have symmetry properties. This issue is well understood in the semilinear case $p=2$ when $\Gamma=\emptyset$. The symmetry of the solutions in this case can be deduced by the celebrated \emph{moving plane method}, see \cite{A,BN,GNN,serrin}. In \cite{EFS, io} the moving plane procedure has been adapted to the case when the singular set has zero capacity, in the semilinear setting $p=2$. \\

\noindent The moving plane procedure has been developed for problems involving the $p$-Laplace operator, in the standard case $\Gamma=\emptyset$, in \cite{DamPac} for $1<p<2$ and in \cite{DSJDE} for $p>2$. In fact, in our proofs, we shall borrow many techniques and ideas from \cite{DamPac,DSJDE} and from \cite{EFS, io}. However the techniques cannot be applied straightforwardly manly for two reasons. First of all the technique in \cite{EFS, io}, that works the case $p=2$, is strongly based on a homogeneity argument that fails for $p\neq2$. Furthermore, since the gradient of the solution may blows up near the critical set, then the equation may exhibit both a degenerate and a singular nature at the same time. This causes in particular that it is no longer true that the case $1<p<2$ allows to get stronger results in a easier way, as it is in the case $\Gamma=\emptyset$. In fact, having in mind this remark, we prefer to start the presentation of our results with the case $p>2$. We have the following
\begin{thm}\label{main}
Let $p>2$ and
let $u\in C^{1}({\overline \Omega}\setminus \Gamma)$ be a solution to \eqref{problem} and assume that $f$ is locally Lipschitz continuous with $f(s)>0$ for $s>0$, namely assume  ($A_f^2$).
If $\Omega$ is convex and symmetric with respect to the $x_1$-direction,
$\Gamma$  is closed with
${\operatorname{Cap}_p}(\Gamma)=0$, namely let us assume ($A_\Gamma^2$),
and
$$\Gamma\subset\{x\in \Omega\,:\, x_1=0\},$$ then it follows that   $u$
is symmetric with respect to the hyperplane $\{x_1=0\}$ and
increasing in the $x_1$-direction in $\Omega\cap\{x_1<0\}$.
\end{thm}
Although  the technique that we will develop to prove Theorem \ref{main} works for any $p>2$, the result is stated for $2<p\leq n$ since there are no sets of zero $p$-capacity when $p>n$.\\

\noindent Surprisingly the case $1<p<2$ presents more difficulties related to the fact that, as already remarked, the operator may degenerate near the critical set even if $p<2$. We will therefore need an accurate analysis on the behaviour of the gradient of the solution near $\Gamma$. We carry out such analysis exploiting the results of \cite{PQS} (therefore we shall require a growth assumption on the nonlinearity) and a blow up argument. The result is the following:
\begin{thm}\label{mainbis}
Let $1<p<2$ and
let $u\in C^{1}({\overline \Omega}\setminus \Gamma)$ be a solution to \eqref{problem} and assume that $f$ is  locally Lipschitz continuous with $f(s)>0$ for $s>0$ and has subcritical growth, namely let us assume ($A_f^1$).
Assume that  $\Gamma$  is closed and that
 $\Gamma = \{0\}$ for $n=2$, while  $\Gamma\subseteq\mathcal M$ for some compact $C^2$ submanifold
$\mathcal M$ of dimension $m\leq n-k$, with $k \geq \frac{n}{2}$  for $n>2$, see ($A_\Gamma^1$).
Then, if $\Omega$ is convex and symmetric with respect to the $x_1$-direction
and
$$\Gamma\subset\{x\in \Omega\,:\, x_1=0\},$$  it follows that   $u$
is symmetric with respect to the hyperplane $\{x_1=0\}$ and
increasing in the $x_1$-direction in $\Omega\cap\{x_1<0\}$.
\end{thm}

\noindent The paper is organized as follows: we prove some technical results in Section \ref{notations} that we will exploit in
Section  \ref{main2proof} to prove  Theorem \ref{main} and Theorem \ref{mainbis}.

\section{Some Technical Results}\label{notations}
\noindent {\bf Notation.} Generic fixed and numerical constants will
be denoted by $C$ (with subscript in some case) and they will be
allowed to vary within a single line or formula. By $|A|$ we will
denote the Lebesgue measure of a measurable set $A$.\\

For a real number
$\lambda$ we set
\begin{equation}\label{eq:sn2}
\Omega_\lambda=\{x\in \Omega:x_1 <\lambda\}
\end{equation}
\begin{equation}\label{eq:sn3}
x_\lambda= R_\lambda(x)=(2\lambda-x_1,x_2,\ldots,x_n)
\end{equation}
which is the reflection through the hyperplane $T_\lambda :=\{x\in \mathbb R^n :  x_1=
\lambda\}$. Also let
\begin{equation}\label{eq:sn4}
a=\inf _{x\in\Omega}x_1.
\end{equation}
Finally we set
\begin{equation}\label{eq:sn33}
u_\lambda(x)=u(x_\lambda)\,.
\end{equation}
We recall also  the definition of $p$-capacity of a compact set $\mathcal
A\subset \mathbb R^n$. For $1\leq p \leq n$ we define
$\operatorname{Cap}_p(\mathcal A)$ as
\begin{equation}\label{eq:cambio1}
\operatorname{Cap}_p(\mathcal A):=
\inf \left\{ \int_{\mathbb{R}^n} |\nabla \varphi
|^p dx < + \infty \; : \; \varphi \in C^{\infty}_c(\mathbb R^n) \ \ \text{and} \ \ \varphi \geq \chi_{\mathcal A}
\right\},
\end{equation} where $\chi_{\mathcal S}$ denotes the characteristic function of a set $\mathcal S$.
By the invariance under  reflections of \eqref{eq:cambio1}, it follows that
\begin{equation}\label{eq:cambio2}
{\operatorname{Cap}_p}(\Gamma)= {\operatorname{Cap}_p}(R_{\lambda}(\Gamma)).
\end{equation}
Moreover it can be shown that, if
${\operatorname{Cap}_p}(R_{\lambda}(\Gamma))=0$, then we have that
\begin{equation}\label{eq:cambio3}
\operatorname{Cap}_p^{D}(R_{\lambda}(\Gamma))=0,\end{equation}
where $D\subset \mathbb{R}^n$ denotes a bounded subset and with $\operatorname{Cap}_p^{D}(\mathcal A)$ ($\mathcal A\subset D$ a compact set of $\mathbb R^n$) we mean
\begin{equation}\nonumber
\operatorname{Cap}^D_p(\mathcal A):=
\inf \left\{ \int_{D} |\nabla \varphi
|^p dx < + \infty \; : \; \varphi \in C^{\infty}_c(D) \ \ \text{and} \ \ \varphi \geq \chi_{\mathcal A}
\right\}.
\end{equation}

%%%%%%%%%%%%%%%%%%%%%%%%%%%%%%
%%%%%%%%%%%%%%%%%%%%%%%%%%%%

Let $\varepsilon >0$ small and let $\mathcal{B}^{\lambda}_{\epsilon}$ be a $\varepsilon$-neighborhood of $R_{\lambda}(\Gamma)$ .
From \eqref{eq:cambio2} and \eqref{eq:cambio3} it follows that there exists $\varphi_{\varepsilon} \in
C^{\infty}_c(\mathcal{B}^{\lambda}_{\epsilon})$ such that
$\varphi_{\varepsilon} \geq 1$ on $\chi_{R_\lambda(\Gamma)}$
and
\begin{equation}\nonumber
\displaystyle \int_{\mathcal{B}^{\lambda}_{\epsilon}} |\nabla
\varphi_{\varepsilon} |^p dx < \varepsilon.
\end{equation}
To carry on our analysis we  need to construct a function
$\psi_{\varepsilon} \in \ W^{1,p}(\Omega)$ such that
$\psi_{\varepsilon} = 1$ in
$\Omega\setminus\mathcal{B}_{\varepsilon}^{\lambda}$,
$\psi_{\varepsilon} = 0$ in a $\delta_{\varepsilon}$-neighborhood
$\mathcal{B}_{\delta_{\varepsilon}}^{\lambda}$  of
$R_\lambda(\Gamma)$ (with $\delta_\varepsilon<\varepsilon$) and such that
\begin{equation}\label{eq:cambio4}
\int_{\mathcal{B}^{\lambda}_{\epsilon}} |\nabla \psi_{\varepsilon}
|^p dx \leq C \varepsilon,\end{equation}
for some positive constant $C$ that does not depend on $\varepsilon$. To construct such a test function we
consider the   real functions $T: \R \rightarrow \R^+_0$ and  $g: \R^+_0 \rightarrow \R^+_0$ defined by
\begin{equation}\label{eq:functionsdefinition}
T(s):=\max \{0;\min\{s;1\}\},\, s\in \mathbb{R}\quad \text{and}\quad g(s):=\max\{0;-2s+1\},\,s\in\mathbb{R}^+_0.
%T(s)= \begin{cases}
%0 &  \text{if}\,\, s\leq0  \\
%s & \text{if}\,\, 0 < s < 1 \\
%1 & \text{if}\,\, s \geq1,
%\end{cases}
\end{equation}
%and
%\begin{equation}\nonumber
%g(s)=\begin{cases}
%-2s+1 &  \text{if}\,\, 0 \leq s < \frac{1}{2}\\
%0 & \text{if}\,\, s
%\geq \frac{1}{2}.
%\end{cases}
%\end{equation}
Finally we set
\begin{equation} \label{test1}
\psi_{\varepsilon}(x):= g(T(\varphi_{\varepsilon}(x))).
\end{equation}
By the definitions \eqref{eq:functionsdefinition}, it follows that
$\psi_{\varepsilon}$ satisfies \eqref{eq:cambio4}.

\noindent To simplify the presentation we summarize the assumptions of the main results as follows:

\noindent {\bf ($A_f^1$)}. {\em For $1<p<2$
we assume that $f$ is locally Lipschitz continuous so that, for any $0\leq t, s\leq M$, there exists a
positive constant $K_f=K_f(M)$ such that
\begin{equation}\nonumber
 |f(s)-f(t)|\leq K_f |s-t|.
\end{equation}
\noindent Moreover $f(s)>0$ for $s>0$ and
$$\lim_{t \rightarrow + \infty} \frac{f(t)}{t^q}=l \in (0,+\infty).$$
 for some $q\in \mathbb R$ such that  $p-1<q<p^*-1$, where $p^*=np/(n-p)$.}

\

\noindent {\bf  ($A_f^2$)}. {\em For $p\geq 2$ we only assume that $f$ is locally Lipschitz continuous so that,
for $0\leq t, s\leq M$ there exists a
positive constant $K_f=K_f(M)$ such that
\begin{equation}\nonumber
 |f(s)-f(t)|\leq K_f |s-t|.
\end{equation}
\noindent Furthermore $f(s)>0$ for $s>0$.}

\noindent {\bf ($A_\Gamma^1$)}. {\em For $1<p<2$ and $n=2$ we assume that $\Gamma = \{0\}$, while for $1<p<2$ and $n>2$ we assume  that $\Gamma\subseteq\mathcal M$ for some compact $C^2$ submanifold
$\mathcal M$ of dimension $m\leq n-k$, with $k \geq \frac{n}{2}$.}

\

\noindent {\bf ($A_\Gamma^2$)}. {\em For $2<p<n$ and $n \geq 2$, we
assume that  $\Gamma$  closed and such that
$${\operatorname{Cap}_p}(\Gamma)=0.$$}
\begin{rem}
We want just to remark that in the case $1 < p < 2$ and $N
> 2$ if $\Gamma\subseteq\mathcal M$ for  some compact $C^2$ submanifold
$\mathcal M$ of dimension $m\leq n-k$ then
${\operatorname{Cap}_p}(\Gamma)=0.$ In this case we consider
$\mathcal{B}_\varepsilon$ a tubular neighborhood of radius
$\varepsilon$ of $\mathcal{M}$, i.e.
$$\mathcal{B}_\varepsilon:=\{x \in \Omega \ : \ \text{dist}(x,\mathcal{M}) < \varepsilon \},$$
with  $\varepsilon > 0$
sufficiently small so that $\mathcal{M}$ has the unique nearest
point property in the neighborhood of $\mathcal{M}$ of radius
$\varepsilon$. We may and do also assume that Fermi coordinates are
well defined in such neighborhood, see e.g. \cite{PPS}. Therefore, using the definition  \eqref{eq:cambio1} above, it can be shown that ${\operatorname{Cap}_p}(\Gamma)=0.$
\end{rem}

Moreover (see for example \cite{lucio}) in the following we further use the following  inequalities: $\forall \eta, \eta' \in  \mathbb{R}^{n}$ with $|\eta|+|\eta'|>0$
there exist positive constants $C_1, C_2,C_3,C_4$ depending on $p$ such that
\begin{equation}\label{eq:inequalities}
\begin{split}
[|\eta|^{p-2}\eta-|\eta'|^{p-2}\eta'][\eta- \eta'] &\geq C_1
(|\eta|+|\eta'|)^{p-2}|\eta-\eta'|^2, \\ \\
\|\eta|^{p-2}\eta-|\eta'|^{p-2}\eta '| & \leq C_2
(|\eta|+|\eta'|)^{p-2}|\eta-\eta '|,\\\\
[|\eta|^{p-2}\eta-|\eta'|^{p-2}\eta '][\eta-\eta '] & \geq C_3
|\eta-\eta '|^p \quad \quad\mbox{if}\quad p\geq 2,\\\\
\|\eta|^{p-2}\eta-|\eta'|^{p-2}\eta '| & \leq C_4 |\eta-\eta
'|^{p-1} \quad\mbox{if}\quad 1 < p \leq 2.
\end{split}
\end{equation}

In the following we will  exploit the fact that
$u_\lambda$ (in the sense of Definition \ref{def:sol}) is a solution
to
\begin{equation}\label{debil2}
\int_{R_\lambda(\Omega)}|\nabla u_\lambda|^{p-2} (\nabla u_\lambda,
\nabla \varphi)\,dx\,=\,\int_{R_\lambda(\Omega)}
f(u_\lambda)\varphi\,dx\qquad\forall \varphi\in
C^{1}_c(R_\lambda(\Omega)\setminus R_\lambda(\Gamma))\,.
\end{equation}
We set
\begin{equation}\label{eq:thankyouuuuuu}
w_\lambda(x)\,:=\,(u-u_\lambda)(x),\quad x\in \overline\Omega_{\lambda} \setminus  R_\lambda(\Gamma).\end{equation}

\begin{lem}\label{leaiuto}
Let  $p>1$ and let $u$ and $u_\lambda$ be solutions to \eqref{debil1} and \eqref{debil2} respectively and let $f: \R \rightarrow \R$ be a
locally Lipschitz continuous function.
Let us assume  $\Gamma\subset \Omega$  closed and such that
$${\operatorname{Cap}_p}(\Gamma)=0.$$
Let $a$ be defined as  in \eqref{eq:sn4}
and
$a<\lambda<0$.

Then
\begin{equation}\nonumber
\int_{\Omega_\lambda}(|\nabla u| + |\nabla u_\lambda|)^{p-2} |\nabla
w_\lambda^+|^2\,dx\leq
C(p,\lambda,\|u\|_{L^\infty(\Omega_\lambda)})\,.
\end{equation}
\end{lem}

\begin{proof}
In all the proof, according to our assumptions, we assume that
  $0\leq t, s\leq M$, there exists a
positive constant $K_f=K_f(M)$ such that
\begin{equation}\nonumber
 |f(s)-f(t)|\leq K_f |s-t|.
\end{equation}
For $\psi_\varepsilon$ defined as in \eqref{test1}, we consider
$$\varphi_\varepsilon := w_\lambda^+ \psi_\varepsilon^p \chi_{\Omega_\lambda}.$$
By standard arguments, since $w_\lambda^+\leq \|u\|_{L^\infty(\Omega_\lambda)}$ (recall that in particular  $u\in C(\overline{\Omega}\setminus \Gamma)$) and by construction $0\leq \psi_\varepsilon\leq 1$, we have that $\varphi_\varepsilon \in
W^{1,p}_0(\Omega_\lambda)$. By a density argument we use  $\varphi_\varepsilon$ as test function
 in \eqref{debil1} and \eqref{debil2}.
Subtracting we get
\begin{equation} \label{starteq}
\begin{split} &\int_{\Omega_\lambda} (|\nabla u|^{p-2} \nabla u
- |\nabla u_\lambda|^{p-2}\nabla u_\lambda,
 \nabla w^+_\lambda) \psi_\varepsilon^p \, dx\\
&+ p \int_{\Omega_\lambda} (|\nabla u|^{p-2} \nabla u - |\nabla u_\lambda|^{p-2} \nabla u_\lambda, \nabla \psi_\varepsilon)
\psi_\varepsilon^{p-1} w_\lambda^+ \, dx \\
&= \int_{\Omega_\lambda} (f(u)-f(u_\lambda)) w_\lambda^+
\psi_\varepsilon^p \, dx
\end{split}
\end{equation}
Now it is useful to split the set $\Omega_\lambda$ as the union of
two disjoint subsets $A_\lambda$ and $B_\lambda$ such that $\Omega_\lambda= A_\lambda \cup B_\lambda$. In particular, for $ \dot C>1$ that will be fixed large,   we set
\begin{equation}\nonumber
A_\lambda = \{ x \in \Omega_\lambda \ :  \ | \nabla u_\lambda (x)| <  \dot C| \nabla u (x)|  \}\quad \text{and}\quad
B_\lambda = \{ x \in \Omega_\lambda \ :  \ | \nabla u_\lambda (x)| \geq  \dot C| \nabla u (x)|\}.
\end{equation}
Then it follows that
\begin{itemize}
\item[-] By  the definition of $ A_\lambda$ it follows that there exists $\hat C$ such that
\begin{equation}\label{eq:montalbano}
 |\nabla u| + |\nabla u_\lambda| <\hat C|\nabla u|.\end{equation}
\item[-] By the definition  of the  set $B_\lambda$ and standard triangular
inequalities, we can deduce the existence of a positive constant $\check C$ such that
\begin{equation} \label{auxiliar}
\frac{1}{\check C} |\nabla u_\lambda| \leq |\nabla
u_\lambda| - |\nabla u| \leq |\nabla w_\lambda| \leq |\nabla
u_\lambda| + |\nabla u| \leq \check C |\nabla
u_\lambda|.
\end{equation}
\end{itemize}

We distinguish two cases:

\

\noindent {\bf {Case 1: $1<p<2$.}} From \eqref{starteq}, using
\eqref{eq:inequalities} and {\bf  ($A_f^1$)} we have
\begin{equation}\label{eq:cambio5}
\begin{split}
&C_1 \int_{\Omega_\lambda} (|\nabla u| + |\nabla u_\lambda|)^{p-2}
|\nabla w_\lambda^+|^2 \psi_\varepsilon^p \, dx  \leq
\int_{\Omega_\lambda} (|\nabla u|^{p-2} \nabla u - |\nabla
u_\lambda|^{p-2}\nabla u_\lambda, \nabla w^+_\lambda) \psi_\varepsilon^p \, dx\\
&\leq p \int_{\Omega_\lambda} \left\|\nabla u|^{p-2} \nabla u -
|\nabla u_\lambda|^{p-2} \nabla u_\lambda \right| \ |\nabla
\psi_\varepsilon| \psi_\varepsilon^{p-1} w_\lambda^+ \, dx+
\int_{\Omega_\lambda} \frac{f(u)-f(u_\lambda)}{u-u_\lambda} (w_\lambda^+)^2 \psi_\varepsilon^p \, dx\\
& \leq p C_4 \int_{\Omega_\lambda} |\nabla w^+_\lambda|^{p-1} \
|\nabla \psi_\varepsilon| \psi_\varepsilon^{p-1} w_\lambda^+ \,
dx+K_f\int_{\Omega_\lambda} (w_\lambda^+)^2 \psi_\varepsilon^p \, dx
\\
&\leq C\Big (I_1+I_2\Big ) + C\int_{\Omega_\lambda}
\psi_\varepsilon^p \, dx,
\end{split}
\end{equation}
where
$$I_1:= \int_{A_\lambda}  |\nabla w^+_\lambda|^{p-1} \ |\nabla
\psi_\varepsilon| \psi_\varepsilon^{p-1} w_\lambda^+ \, dx,$$
$$I_2:=\int_{B_\lambda} |\nabla w^+_\lambda|^{p-1} \ |\nabla
\psi_\varepsilon| \psi_\varepsilon^{p-1} w_\lambda^+ \, dx,$$ and
$C=C(p, \lambda ,\|u\|_{L^{\infty}(\Omega_\lambda)})$ is a positive
constant.

\emph{Step 1: Evaluation of $I_1$.} Using  Young's inequality and
\eqref{eq:montalbano}, we have
\begin{equation}\label{eq:asjdhjaskhdjbjhbhjwefohi}
\begin{split}
{I}_1 &= \int_{A_\lambda} |\nabla w^+_\lambda|^{p-1} |\nabla
\psi_\varepsilon| \psi_\varepsilon^{p-1} w_\lambda^+ \,  dx
\leq
\left(\int_{A_\lambda} |\nabla w_\lambda^+|^p \psi_\varepsilon^p
\, dx \right)^{\frac{p-1}{p}}
  \left(\int_{A_\lambda} |\nabla \psi_\varepsilon|^p
(w_\lambda^+)^p \, dx\right)^{\frac{1}{p}}\\& \leq
\left(\int_{A_\lambda} (|\nabla u|+|\nabla
u_\lambda|)^p \psi_\varepsilon^p \, dx \right)^{\frac{p-1}{p}} \left( \int_{A_\lambda} |\nabla
\psi_\varepsilon|^p (w_\lambda^+)^p \, dx \right)^{\frac{1}{p}}\\& \leq \left(\hat C
\int_{A_\lambda}
|\nabla u|^p \psi_\varepsilon^p \, dx  \right)^{\frac{p-1}{p}}\left( \int_{A_\lambda} |\nabla \psi_\varepsilon|^p (w_\lambda^+)^p \, dx \right)^{\frac{1}{p}}\\
&\leq C \left(\int_{\Omega_\lambda}|\nabla u|^p\,dx \right)^{\frac{p-1}{p}}\left(
 \int_{\Omega_\lambda} |\nabla \psi_\varepsilon|^p \, dx \right)^{\frac{1}{p}},
\end{split}
\end{equation}
where $C=C(p,\lambda, \|u\|_{L^{\infty}(\Omega_\lambda)})$ is   a positive
constant.

\emph{Step 2: Evaluation of $I_2$.} Using the weighted Young's inequality
and \eqref{auxiliar} we get
\begin{equation}\label{eq:asdadajdhjaskhdjbjhbhjwefohi}
\begin{split}
&{I}_2= \int_{B_\lambda} |\nabla w^+_\lambda|^{p-1} |\nabla
\psi_\varepsilon|  \psi_\varepsilon^{p-1} w_\lambda^+ \, dx
\leq  \delta \int_{B_\lambda} |\nabla w_\lambda^+|^p
\psi_\varepsilon^p \, dx + \frac{1}{\delta} \int_{B_\lambda}
|\nabla\psi_\varepsilon|^p (w_\lambda^+)^p \, dx
\\ & \leq  \delta\int_{B_\lambda} (|\nabla u| +
|\nabla u_\lambda|)^{p-2} \left(|\nabla u| + |\nabla
u_\lambda|\right)^2 \psi_\varepsilon^p \, dx + \frac{1}{\delta}
\int_{B_\lambda} |\nabla\psi_\varepsilon|^p (w_\lambda^+)^p \,
dx\\& \leq  \delta\check C^2  \int_{B_\lambda} (|\nabla u| + |\nabla
u_\lambda|)^{p-2} |\nabla u_\lambda|^2 \psi_\varepsilon^p \, dx +
\frac{1}{\delta} \int_{B_\lambda} |\nabla\psi_\varepsilon|^p
(w_\lambda^+)^p \, dx\\& \leq  \delta\check C^4 \int_{B_\lambda}
(|\nabla u| + |\nabla u_\lambda|)^{p-2} |\nabla w_\lambda^+|^2
\psi_\varepsilon^p \, dx + \frac{1}{\delta} \int_{B_\lambda}
|\nabla \psi_\varepsilon|^p (w_\lambda^+)^p \, dx\\ & \leq \delta
C \int_{\Omega_\lambda} (|\nabla u| + |\nabla u_\lambda|)^{p-2}
|\nabla w^+_\lambda|^2 \psi_\varepsilon^p \, dx +
\frac{C}{\delta} \int_{\Omega_\lambda} |\nabla
\psi_\varepsilon|^p \, dx,
\end{split}
\end{equation}
where $C=C(p,\lambda, \|u\|_{L^{\infty}(\Omega_\lambda)})$ is   a positive
constant.  Finally, using \eqref{eq:cambio5},
\eqref{eq:asjdhjaskhdjbjhbhjwefohi} and
\eqref{eq:asdadajdhjaskhdjbjhbhjwefohi}, we obtain
\begin{equation}\label{eq:cambio6}
\begin{split}
& \int_{\Omega_\lambda} (|\nabla u| + |\nabla u_\lambda|)^{p-2}
|\nabla w_\lambda^+|^2 \psi_\varepsilon^p \, dx  \\& \leq \delta C
\int_{\Omega_\lambda} (|\nabla u| + |\nabla u_\lambda|)^{p-2}
|\nabla w_\lambda^+|^2\psi_\varepsilon^p \, dx  +C
\left(\int_{\Omega_\lambda}|\nabla u|^p\,dx \right)^{\frac{p-1}{p}}\left(
 \int_{\Omega_\lambda} |\nabla \psi_\varepsilon|^p \, dx \right)^{\frac{1}{p}}\\
 &+\frac{C}{ \delta}
\int_{\Omega_\lambda} |\nabla \psi_\varepsilon|^p \, dx
+ C\int_{\Omega_\lambda} \psi_\varepsilon^p \, dx,
\end{split}
\end{equation}
for some positive constant $C=C(p,
\lambda,\|u\|_{L^{\infty}(\Omega_\lambda)})$.

\

\noindent {\bf {Case 2: $p\geq 2$.}} From \eqref{starteq}, using
\eqref{eq:inequalities} and {\bf  ($A_f^2$)} we have
\begin{equation}\label{eq:estp2}
\begin{split}
&C_1 \int_{\Omega_\lambda} (|\nabla u| + |\nabla u_\lambda|)^{p-2}
|\nabla w_\lambda^+|^2 \psi_\varepsilon^p \, dx  \leq
\int_{\Omega_\lambda} (|\nabla u|^{p-2} \nabla u - |\nabla u_\lambda|^{p-2}\nabla u_\lambda, \nabla w^+_\lambda) \psi_\varepsilon^p \, dx\\
&=-p\int_{\Omega_\lambda} (|\nabla u|^{p-2} \nabla u - |\nabla
u_\lambda|^{p-2} \nabla u_\lambda, \nabla \psi_\varepsilon)
\psi_\varepsilon^{p-1} w_\lambda^+ \, dx + \int_{\Omega_\lambda}
(f(u)-f(u_\lambda))
w_\lambda^+ \psi_\varepsilon^p \, dx\\
&\leq p \int_{\Omega_\lambda} \left\|\nabla u|^{p-2} \nabla u -
|\nabla u_\lambda|^{p-2} \nabla u_\lambda \right| \ |\nabla
\psi_\varepsilon| \psi_\varepsilon^{p-1} w_\lambda^+ \, dx+
\int_{\Omega_\lambda}
\frac{f(u)-f(u_\lambda)}{u-u_\lambda} (w_\lambda^+)^2 \psi_\varepsilon^p \, dx\\
& \leq p C_2 \int_{\Omega_\lambda} (|\nabla u| + |\nabla
u_\lambda|)^{p-2} |\nabla w^+_\lambda| \ |\nabla \psi_\varepsilon|
\psi_\varepsilon^{p-1} w_\lambda^+ \, dx+ K_f \int_{\Omega_\lambda}
(w_\lambda^+)^2 \psi_\varepsilon^p \, dx
\\
&=p C_2 \int_{A_\lambda} (|\nabla u| + |\nabla
u_\lambda|)^{p-2} |\nabla w^+_\lambda| \ |\nabla
\psi_\varepsilon| \psi_\varepsilon^{p-1} w_\lambda^+ \,
dx\\
&+p C_2 \int_{B_\lambda} (|\nabla u| + |\nabla u_\lambda|)^{p-2}
|\nabla w^+_\lambda| \ |\nabla \psi_\varepsilon|
\psi_\varepsilon^{p-1} w_\lambda^+ \, dx+ K_f \int_{\Omega_\lambda}
(w_\lambda^+)^2
\psi_\varepsilon^p \, dx\\
&\leq C\Big (I_1+I_2\Big ) + C\int_{\Omega_\lambda}
\psi_\varepsilon^p \, dx,
\end{split}
\end{equation}
where
$$I_1:= \int_{A_\lambda} (|\nabla u| + |\nabla
u_\lambda|)^{p-2} |\nabla w^+_\lambda| \ |\nabla
\psi_\varepsilon| \psi_\varepsilon^{p-1} w_\lambda^+ \,
dx,$$
$$I_2:=\int_{B_\lambda} (|\nabla u| + |\nabla
u_\lambda|)^{p-2} |\nabla w^+_\lambda| \ |\nabla
\psi_\varepsilon| \psi_\varepsilon^{p-1} w_\lambda^+ \,
dx,$$
and $C=C(p, \lambda ,\|u\|_{L^{\infty}(\Omega_\lambda)})$ is a positive constant.

\emph{Step 1: Evaluation of $I_1$.} Using the weighted Young's inequality we have
\[ \nonumber
\begin{split} I_1&= \int_{A_\lambda} (|\nabla u| + |\nabla
u_\lambda|)^{p-2} |\nabla w^+_\lambda| \ |\nabla
\psi_\varepsilon| \psi_\varepsilon^{p-1} w_\lambda^+ \,
dx\\
&\leq \delta \int_{A_\lambda} (|\nabla u| + |\nabla
u_\lambda|)^{p-2} |\nabla w_\lambda^+|^2 \psi_\varepsilon^p
\, dx+ \frac{1}{\delta} \int_{A_\lambda} (|\nabla u| + |\nabla
u_\lambda|)^{p-2} |\nabla \psi_\varepsilon|^2
\psi_\varepsilon^{p-2} (w_\lambda^+)^2 \, dx.
\end{split}
\]
%Now using the elementary inequality
%$$(|\nabla u| +
%|\nabla u_\lambda|)^{p-2} \leq 2^{p-2} (|\nabla u|^{p-2} + |\nabla
%u_\lambda|^{p-2})$$ and that
Using \eqref{eq:montalbano} and  H\"older inequality, we obtain
\begin{equation}\label{eq:cambio7}
\begin{split}
I_1 &
\leq \delta \int_{A_\lambda} (|\nabla u| + |\nabla u_\lambda|)^{p-2}
|\nabla w_\lambda^+|^2 \psi_\varepsilon^p \, dx +
\frac{\hat C^{p-2}}{\delta}\int_{A_\lambda} |\nabla u|^{p-2}
 |\nabla \psi_\varepsilon|^2 \psi_\varepsilon^{p-2}
(w_\lambda^+)^2 \, dx\\
&\leq \delta \int_{A_\lambda} (|\nabla u| + |\nabla
u_\lambda|)^{p-2} |\nabla w_\lambda^+|^2 \psi_\varepsilon^p
\, dx  +  \frac{C}{\delta} \left(\int_{A_\lambda} |\nabla u|^p \psi_\varepsilon^p
\, dx \right)^{\frac{p-2}{p}} \left(\int_{A_\lambda} |\nabla
\psi_\varepsilon|^p (w_\lambda^+)^p \, dx
\right)^{\frac{2}{p}}
\\
&\leq \delta \int_{\Omega_\lambda} (|\nabla u| + |\nabla
u_\lambda|)^{p-2} |\nabla w_\lambda^+|^2 \psi_\varepsilon^p
\, dx  +  \frac{C}{\delta} \left(\int_{\Omega_\lambda} |\nabla u|^p
\, dx \right)^{\frac{p-2}{p}} \left(\int_{\Omega_\lambda} |\nabla
\psi_\varepsilon|^p \, dx
\right)^{\frac{2}{p}},
\end{split}
\end{equation}
with $C=C(p,\lambda, \|u\|_{L^{\infty}(\Omega_\lambda)})$ is   a positive constant.

\emph{Step 2: Evaluation of $I_2$}. By the weighted Young's inequality
\begin{equation} \nonumber
\begin{split}
I_2&:=\int_{B_\lambda} (|\nabla u| + |\nabla u_\lambda|)^{p-2}
|\nabla w^+_\lambda| \ |\nabla \psi_\varepsilon|
\psi_\varepsilon^{p-1} w_\lambda^+ \, dx\\
& \leq \delta \int_{B_\lambda} (|\nabla u| + |\nabla
u_\lambda|)^{\frac{p(p-2)}{p-1}} |\nabla
w_\lambda^+|^{\frac{p}{p-1}} \psi_\varepsilon^p
\, dx+ \frac{1}{\delta} \int_{B_\lambda} |\nabla \psi_\varepsilon|^p (w_\lambda^+)^p \, dx\\
&= \delta \int_{B_\lambda} (|\nabla u| + |\nabla
u_\lambda|)^{\frac{p(p-2)}{p-1}} |\nabla w_\lambda^+|^2 |\nabla
w_\lambda^+|^{\frac{p}{p-1}-2} \psi_\varepsilon^p
\, dx+ \frac{1}{\delta} \int_{B_\lambda} |\nabla \psi_\varepsilon|^p
(w_\lambda^+)^p \, dx.
\end{split}
\end{equation}
Using \eqref{auxiliar} and noticing that
\[\frac{p}{(p-1)}-2 < 0,\] we obtain the following
estimate
\begin{equation}\label{eq:gunssssss&rosesssss}
\begin{split}
I_2 &\leq \delta \check C^{\frac{(p-2)(p+1)}{p-1}}\int_{B_\lambda} |\nabla
u_\lambda|^{p-2} |\nabla w_\lambda^+|^2 \psi_\varepsilon^p
\, dx+ \frac{1}{\delta} \int_{B_\lambda} |\nabla \psi_\varepsilon|^p (w_\lambda^+)^p \, dx
\\
& \leq \delta{C} \int_{B_\lambda} (|\nabla u| + |\nabla
u_\lambda|)^{p-2} |\nabla w_\lambda^+|^2 \psi_\varepsilon^p
\, dx + \frac{C}{\delta}
\int_{B_\lambda} |\nabla \psi_\varepsilon|^p \, dx\\
& \leq \delta{C} \int_{\Omega_\lambda} (|\nabla u| + |\nabla
u_\lambda|)^{p-2} |\nabla w_\lambda^+|^2 \psi_\varepsilon^p
\, dx + \frac{C}{\delta}
\int_{\Omega_\lambda} |\nabla \psi_\varepsilon|^p \, dx,
\end{split}
\end{equation}
with $C=C(p, \|u\|_{L^{\infty}(\Omega_\lambda)})$. In the second line of \eqref{eq:gunssssss&rosesssss} we exploited the fact that, since $p\geq 2$ then \[|\nabla
u_\lambda|^{p-2}\leq (|\nabla u| + |\nabla
u_\lambda|)^{p-2}. \]
Collecting
\eqref{eq:estp2}, \eqref{eq:cambio7} and
\eqref{eq:gunssssss&rosesssss}  we
deduce that
\begin{equation}\label{eq:cambio8}
\begin{split}
& \int_{\Omega_\lambda} (|\nabla u| + |\nabla u_\lambda|)^{p-2}
|\nabla w_\lambda^+|^2 \psi_\varepsilon^p \, dx\\
 &\leq \delta C \int_{\Omega_\lambda} (|\nabla u| + |\nabla
u_\lambda|)^{p-2} |\nabla w_\lambda^+|^2 \psi_\varepsilon^p
\, dx  +  \frac{C}{\delta} \left(\int_{\Omega_\lambda} |\nabla u|^p
\, dx \right)^{\frac{p-2}{p}} \left(\int_{\Omega_\lambda} |\nabla
\psi_\varepsilon|^p \, dx
\right)^{\frac{2}{p}}\\
& + \frac{C}{\delta}
\int_{\Omega_\lambda} |\nabla \psi_\varepsilon|^p \, dx+C\int_{\Omega_\lambda}
\psi_\varepsilon^p \, dx,
\end{split}
\end{equation}
for some positive constant $C=C(p, \lambda,\|u\|_{L^{\infty}(\Omega_\lambda)})$.

\

For $\delta$ small, from \eqref{eq:cambio6} and \eqref{eq:cambio8}, using \eqref{eq:cambio4} and the fact that  for $\lambda<0$ the solution $u\in W^{1,p}(\Omega_\lambda)$,  letting $\varepsilon \rightarrow 0$ by Fatou's Lemma  we obtain
$$\int_{\Omega_\lambda} (|\nabla u| + |\nabla u_\lambda|)^{p-2} |\nabla w_\lambda^+|^2 \, dx
\leq C(p,\lambda,\|u\|_{L^\infty(\Omega_\lambda)}),$$
concluding the proof.
\end{proof}
\section{Proof of Theorem \ref{main} and Theorem \ref{mainbis}}\label{main2proof}
\noindent We recall the fact that $u_\lambda$ (in the sense of
Definition \ref{def:sol}) is a solution to
\begin{equation}\label{debil150}
\int_{R_\lambda(\Omega)}|\nabla u_\lambda|^{p-2} (\nabla u_\lambda,
\nabla \varphi)\,dx\,=\,\int_{R_\lambda(\Omega)}
f(u_\lambda)\varphi\,dx\qquad\forall \varphi\in
C^{1}_c(R_\lambda(\Omega)\setminus R_\lambda(\Gamma))\,.
\end{equation}
We set
\[
w_\lambda(x)\,:=\,(u-u_\lambda)(x),\quad x\in \overline\Omega\setminus (\Gamma\cup R_\lambda(\Gamma)).\]

\noindent Since in the following we will exploit weighted Sobolev inequalities, it is convenient to set
weight
\begin{equation}\label{weightt}
\hat \varrho:=|\nabla u|^{p-2}\qquad
\frac{1}{\hat \varrho}:=|\nabla u|^{2-p}\,.
\end{equation}
We have the following
\begin{lem}\label{sumweight}
Let $1<p<2$. Under the same assumption of Theorem \ref{mainbis}, define
\[\Omega_\lambda^+:=\Omega_\lambda \cap \text{supp}\,(w_\lambda^+).\]
Then
\begin{equation}\label{eq:sumweight}
|\nabla u |^{2-p} \in L^t(R_\lambda(\Omega_\lambda^+)),
\end{equation}
for some $t > \frac{n}{2}$.
\end{lem}

\begin{proof}
By definition of $\Omega_\lambda^+$ we have
$$\|u\|_{L^\infty(R_\lambda(\Omega_\lambda^+))} =
\|u_\lambda\|_{L^\infty(\Omega_\lambda^+)} \leq
\|u\|_{L^\infty(\Omega_\lambda^+)} \leq C(\lambda, \|u\|_{L^\infty(\Omega_\lambda)}).$$
Taking $x_0 \in R_\lambda(\Omega_\lambda^+)\setminus \Gamma$, we
set:
\begin{equation}\label{eq:scaling} g(x):=u(dx+x_0) \quad \text{in} \
B_{\frac{1}{2}}(0),
\end{equation}
where $d:=\text{dist}(x_0,\Gamma)$. Since $u$ is a solution (in the
sense of Definition \ref{def:sol}) to \eqref{problem}, we deduce
that for any $\varphi \in C^1_c(B_{1/2}(0))$
\begin{eqnarray}\label{eq:fantozzi}\\\nonumber
&&\int_{B_{\frac{1}{2}}(0)}|\nabla g|^{p-2}(\nabla g, \nabla \varphi)\, dx\\\nonumber &&=
d^{p-1}\int_{B_{\frac{1}{2}}(0)}|\nabla u(dx+x_0)|^{p-2}(\nabla u(dx+x_0), \nabla \varphi)\, dx\\\nonumber
&&=d^{p-n}\int_{B_{\frac{d}{2}}(x_0)}|\nabla u(x)|^{p-2}(\nabla u(x),\nabla (\varphi (\frac{x-x_0}{d}) ))\, dx=d^{p-n}\int_{B_{\frac{d}{2}}(x_0)} f(u(x))\varphi (\frac{x-x_0}{d})\,dx\\\nonumber
&&=d^{p}\int_{B_{\frac{1}{2}}(0)} f(u(dx+x_0))\varphi (x)\,dx=\int_{B_{\frac{1}{2}}(0)}c(x)(g(x))^{p-1}\varphi (x)\,dx,
\end{eqnarray}
with
\begin{equation}\label{eq:popopo}
c(x):= d^p
\frac{f(u(dx+x_0))}{u^{p-1}(dx+x_0)}.\end{equation}
From \eqref{eq:fantozzi} we deduce that in distributional  sense
\[-\Delta_p g= c(x)g^{p-1}\quad \text{in}\,\, B_{\frac{1}{2}}(0).\]
On the other hand   $u$ as well (in distributional sense) is a positive solution to  $-\Delta_p u=f(u)$  in $B_d(x_0)$. Therefore using   \cite[Theorem 3.1]{PQS} we have
\begin{equation}\label{eq:guhball}
0<u(x) \leq C (1+d^{-\frac{p}{q+1-p}}),
\end{equation}
where $C=C(f,n,p)>0$. By \eqref{eq:popopo}, using {\bf($A_f^1$)}  we have
\begin{equation}\label{eq:neveeeee}
c(x)=Cd^p(1+u^{q+1-p}),
\end{equation}
with $C=C(l,p,K_f)$ is a positive constant. Finally, collecting \eqref{eq:guhball} and \eqref{eq:neveeeee} we deduce
\begin{eqnarray}\nonumber
c(x) \leq C d^p  ( 1+
d^{-p})\leq C,
\end{eqnarray}
with $C=C(f,l,n,p,q,K_f,\Omega)$.
Hence $c(x) \in L^\infty(B_{{1}/{2}}(0))$.
By
\cite[Theorem $7.2.1$]{PucciSerrin}, recalling \eqref{eq:scaling},  for every $x\in B_{1/8}(0)$ it follows
\begin{eqnarray}\label{eq:harnack}
&&g(x) \leq \sup_{x \in B_{\frac{1}{4}}(0)} g(x) \\\nonumber
&&\leq C_H
\inf_{x \in B_{\frac{1}{4}}(0)} g(x) \leq C_H g(0) \leq C
\end{eqnarray}
where $C=C(f,l,n,p,q,K_f,\Omega)$ is a positive constant. Hence $g$ is bounded in
$B_{{1/8}}(0)$ and as consequence, see e.g.
\cite{DB, T}   \[g \in C^{1,\alpha}(B_{\frac{1}{16}}(0)).\]  Then there exists a positive constant $C=C(n,p,\lambda, \|u\|_{L^\infty(\Omega_\lambda)})$ such that
$$|\nabla g(x)|  \leq C \quad \forall x \in B_{\frac{1}{16}}(0).$$
By \eqref{eq:scaling} it follows
$$d  |\nabla u(dx+x_0)| \leq C \quad\forall x \in B_{\frac{1}{16}}(0),$$
namely
\begin{equation}\label{eq:playstation}|\nabla u| \leq \frac{C}{d} \quad  \text{in} \,\, B_{\frac{d}{16}}(x_0).
\end{equation}
Using {\bf ($A_\Gamma^1$)},  we can consider
$\mathcal{B}_\varepsilon$ a tubular neighborhood of radius
$\varepsilon$ of $\mathcal{M}$, i.e.
$$\mathcal{B}_\varepsilon:=\{x \in \Omega \ : \ \text{dist}(x,\mathcal{M}) < \varepsilon \}.$$
We now exploit an integration in  Fermi coordinates,see e.g. \cite{PPS}. We indicate a point of $\mathcal{B}_\varepsilon$ via the coordinate $(\sigma,x)'$ where $\sigma$ is the variable describing the manifold $\mathcal M$ and $x'\in\mathbb{R}^k$ is the Euclidean variable on the normal section. For $\sigma$ fixed, $D_\sigma$ will stand for the normal section at $\sigma$.
By \eqref{eq:playstation},  and passing to polar coordinates
we obtain
\begin{equation} \label{eq:weight}
\begin{split}
\int_{R_\lambda(\Omega_\lambda^+)} \left(|\nabla u|^{2-p}\right)^t
\, dx &= \int_{R_\lambda(\Omega_\lambda^+) \setminus
\mathcal{B}_\varepsilon} \left(|\nabla u|^{2-p}\right)^t \, dx +
\int_{\mathcal{B}_\varepsilon} \left(|\nabla u|^{2-p}\right)^t \, dx \\
&\leq C + C\int_{\mathcal{M}} d \sigma \int_{D_\sigma} \left(|\nabla u|^{2-p}\right)^t \, dx'\\
&\leq  C +  C \int_{\mathcal{M}} d \sigma \int_0^{\varepsilon} \frac{1}{r^{(2-p)t-(k-1)}} \, d r\\
&=C(n,p,\lambda, \|u\|_{L^\infty(\Omega_\lambda)}) +CE_1,
\end{split}
\end{equation}
 with
\begin{equation}\label{eq:weight2}
E_1:=\int_{\mathcal{M}} d \sigma \int_0^{\varepsilon}
\frac{1}{r^{(2-p)t-(k-1)}} \, d r < +\infty,
\end{equation}
if  $t < {k}/{(2-p)}$,
recalling that  $1 < p < 2$. Hence, since $k \geq {n}/{2}$, inequality \eqref{eq:weight2} holds for some
$$t \in \left(\frac{n}{2} , \frac{k}{2-p}\right),$$
being  ${2k} > {n(2-p)}$ under our assumption.
\end{proof}

Let us now set
$$\mathcal{Z}_\lambda:=\{x \in \Omega_\lambda  \setminus R_\lambda (\Gamma) \ | \ \nabla u (x) = \nabla u_\lambda (x) = 0\}.$$
We have the following
\begin{lem}\label{conncomp}
Let $u$ be a solution to \eqref{debil1} with $f: \R \rightarrow \R$ be a
locally Lipschitz  function  such that $f(s)>0$ for $s>0$. Let $a < \lambda
<0$. If $C_\lambda \subset \Omega_\lambda \setminus
\left(R_\lambda(\Gamma) \cup \mathcal{Z}_\lambda \right)$ is a
connected component of $\Omega_\lambda \setminus
\left(R_\lambda(\Gamma) \cup \mathcal{Z}_\lambda \right)$ and 
$u=u_\lambda$ in $C_\lambda$, then \[C_\lambda = \emptyset.\]
\end{lem}

\begin{proof}
Let
$$C:=C_\lambda \cup R_\lambda (C_\lambda).$$
Arguing by contradiction we assume $C \neq \emptyset$. Now for $\varepsilon >0$, we define $h_{\varepsilon}(t):\mathbb R^+_0\rightarrow \mathbb R$ as
\begin{equation}\nonumber
h_\varepsilon(t)=\begin{cases}\frac{G_\varepsilon(t)}{t} & \text{if} \,\, t>0\\
0& \text{if} \,\, t=0,
\end{cases}
\end{equation} where
$G_{\varepsilon}(t)= (2t-2\varepsilon)\chi_{[\varepsilon\,
,\,2\varepsilon]}(t)+t\chi_{[2\varepsilon\, ,\,\infty)}(t)$ for
$t\geqslant 0$. Moreover we consider the cut-off
function $\psi_\varepsilon$ on the set $\Gamma \cup
R_\lambda(\Gamma)$ defined in a similar way as in \eqref{test1}.
Hence we define the test function
$$\varphi_\varepsilon := h_\varepsilon (|\nabla u|) \psi^2_\varepsilon \chi_C.$$
 We point out that the $\text{supp}\, \varphi_\varepsilon \subset C$ and therefore we can use it as test function  in \eqref{debil1}. We obtain
\begin{equation} \nonumber
\begin{split}
0<\int_C f(u)h_\varepsilon (|\nabla u|)\psi^2_\varepsilon \,dx  &= \, \int_C |\nabla u|^{p-2} (\nabla u,
\nabla |\nabla u|) h_{\varepsilon}'(|\nabla u|) \psi_\varepsilon^2 \,dx\,\\
&+ 2\int_C |\nabla u|^{p-2} (\nabla u, \nabla \psi_\varepsilon)
h_{\varepsilon}(|\nabla u|)\psi_\varepsilon
\,dx.
\end{split}
\end{equation}
Using  Schwartz inequality, observing that  \[h_\varepsilon(t)\leq 1\quad \text{and} \quad h'_\varepsilon(t)\leq 2/\varepsilon,\]  we
obtain
\begin{equation} \label{hkbgikbuhjb}
\begin{split}
&0<\int_C f(u)\frac{G_\varepsilon(|\nabla u|)}{|\nabla u|}
\psi^2_\varepsilon \,dx\\  &\leq \, 2\int_{C\cap\{\varepsilon <
|\nabla u| < 2\varepsilon\}}
|\nabla u|^{p-2} \|D^2 u\| \psi_\varepsilon^2  \frac{|\nabla u|}{\varepsilon}   \,dx+ 2 \int_C |\nabla u|^{p-1}|\nabla \psi_\varepsilon|
\psi_\varepsilon \,dx\\
&\leq 4 \int_{C\cap\{\varepsilon <
|\nabla u| < 2\varepsilon\}} |\nabla
u|^{p-2} \|D^2 u\| \psi_\varepsilon^2 \, dx + 2 \int_C |\nabla u|^{p-1}|\nabla \psi_\varepsilon|
\psi_\varepsilon \,dx\\
&\leq4 \int_C |\nabla
u|^{p-2} \|D^2 u\| \psi_\varepsilon^2 \chi_{A_\varepsilon}\, dx + 2 \left(\int_C |\nabla u|^{p}\,dx\right)^{\frac{p-1}{p}}\left(\int_C|\nabla \psi_\varepsilon|^p\,dx\right)^{\frac 1p},
\end{split}
\end{equation}
where $A_\varepsilon:= C \cap\{\varepsilon < |\nabla u| <
2\varepsilon\}$. Now we note that by the definition of the region   $C$ and because $u=u_\lambda$ in $C_\lambda$,  then the solution $u$ is bounded and  $C^{1,\alpha}$ by classical regularity results.
Moreover
\[\displaystyle |\nabla u|^{p-2} \|D^2 u\| \psi_\varepsilon^2
\chi_{A_\varepsilon} \leq |\nabla u|^{p-2} \|D^2 u\|
\] and $\displaystyle |\nabla u|^{p-2} \|D^2 u\|
 \in L^1(\mathcal C)$ by \cite{DSJDE} (see also \cite [Lemma 5]{MMPS} for details).
It is important to note that the regularity of the solution in $R_\lambda (C_\lambda)$ is induced by symmetry by the regularity in $C_\lambda$.
Noticing that
$\displaystyle |\nabla u|^{p-2} \|D^2 u\| \psi_\varepsilon^2
\chi_{A_\varepsilon} \rightarrow 0$ as $\varepsilon$ goes to
$0$, then letting $\varepsilon \rightarrow 0$ in
\eqref{hkbgikbuhjb}, by Dominated Convergence Theorem and \eqref{eq:cambio4} it follows
$$0< \int_C f(u) \, dx \leq 0,$$
and this gives a contradiction. Hence $C=\emptyset$.

\end{proof}

\begin{proof}[Proof of Theorem \ref{mainbis}]
Since the singular set $\Gamma$ is contained in the hyperplane
$\{x_1 = 0\}$, then the moving plane procedure can be started in the
standard way (see e.g. \cite{DamPac, DamSciunzi, DSJDE}) and, for $a
< \lambda < a + \sigma$ with $\sigma>0$ small, we have that
$w_\lambda \leq 0$ in $\Omega_\lambda$ (see
\eqref{eq:thankyouuuuuu}) by the weak comparison principle in small
domains. Note that the crucial point here is that $w_\lambda$ has a
singularity at $\Gamma$ and at $R_\lambda (\Gamma)$. For $\lambda$
close to $a$ the singularity does not play a role. To proceed
further we define
\begin{equation}\nonumber
\Lambda_0=\{a<\lambda<0 : u\leq
u_{t}\,\,\,\text{in}\,\,\,\Omega_t\setminus
R_t(\Gamma)\,\,\,\text{for all $t\in(a,\lambda]$}\}
\end{equation}
and $\lambda_0 = \sup \Lambda_0$, since we proved above that $\Lambda_0$ is not empty. To prove our result we have to
show that $\lambda_0 = 0$.
To do this we
assume that $\lambda_0 < 0$ and we reach a contradiction by proving
that $u \leq u_{\lambda_0 + \tau}$ in $\Omega_{\lambda_0 + \tau}
\setminus R_{\lambda_0 + \tau} (\Gamma)$ for any $0 < \tau <
\bar{\tau}$ for some small $\bar{\tau}>0$. We  remark that
$|\mathcal{Z}_{\lambda_0}|=0$, see  \cite{DSJDE}.
Let us take $A_{\lambda_0}\subset \Omega_{\lambda_0}$ be an open set such that $\mathcal Z_{\lambda_0}\cap\Omega_{\lambda_0}\subset
A_{\lambda_0} \subset \subset \Omega$. Such set exists  by H\"opf's Lemma. Moreover note  that, since $|\mathcal Z_{\lambda_0}|=0$, we can take $A_{\lambda_0}$ of arbitrarily small measure. By
continuity we know that $u \leq u_{\lambda_0}$ in $\Omega_{\lambda_0}
\setminus R_{\lambda_0} (\Gamma)$.
We can exploit the strong comparison principle, see e.g. \cite[Theorem 2.5.2]{PucciSerrin}  to get that, in any connected component of $\Omega_{\lambda_{0}}\setminus \mathcal Z_{\lambda_0}$, we have
$$
u<u_{\lambda_0} \qquad\text{or}\qquad u\equiv u_{\lambda_0}.$$
 The case $u\equiv u_{\lambda_0}$ in some  connected component
$C_{\lambda_{0}}$ of $\Omega_{\lambda_{0}}\setminus \mathcal Z_{\lambda_0}$ is not
possible, since by symmetry, it would imply the existence of  a local symmetry phenomenon and consequently that $\Omega \setminus \mathcal  Z_{\lambda_0}$ would be not connected,  in  spite of what we proved in Lemma \ref{conncomp}. Hence we deduce that $u <
u_{\lambda_0}$ in $\Omega_{\lambda_0} \setminus R_{\lambda_0}
(\Gamma)$. Therefore, given a compact set $K \subset
\Omega_{\lambda_0} \setminus (R_{\lambda_0} (\Gamma)\cup A_{\lambda_0})$, by
uniform continuity we can ensure that $u < u_{\lambda_0+\tau}$ in
$K$ for any $0 < \tau < \bar{\tau}$ for some small $\bar{\tau}>0$.
Note that to do this we implicitly assume, with no loss of
generality, that $R_{\lambda_0} (\Gamma)$ remains bounded away
from K. Arguing in a similar fashion as in Lemma \ref{leaiuto}, we
consider
\begin{equation}\label{eq:moduloooooooo}\varphi_\varepsilon := w^+_{\lambda_0 + \tau} \psi_\varepsilon^p \chi_{\Omega_{\lambda_0 + \tau}}.
\end{equation}
By density arguments as above, we plug $\varphi_\varepsilon$ as test
function in \eqref{debil1} and \eqref{debil150} so that,
subtracting, we get

\begin{equation} \label{eq:Step1}
\begin{split} &\int_{\Omega_{\lambda_0 + \tau} \setminus K} (|\nabla u|^{p-2} \nabla u
- |\nabla u_{\lambda_0 + \tau}|^{p-2}\nabla u_{\lambda_0 + \tau},
 \nabla w^+_{\lambda_0 + \tau}) \psi_\varepsilon^p \, dx\\
&+ p \int_{\Omega_{\lambda_0 + \tau} \setminus K} (|\nabla u|^{p-2}
\nabla u - |\nabla u_{\lambda_0 + \tau}|^{p-2} \nabla u_{\lambda_0 +
\tau}, \nabla \psi_\varepsilon)
\psi_\varepsilon^{p-1} w_{\lambda_0 + \tau}^+ \, dx \\
&= \int_{\Omega_{\lambda_0 + \tau} \setminus K} (f(u)-f(u_\lambda))
w_{\lambda_0 + \tau}^+ \psi_\varepsilon^p \, dx.
\end{split}
\end{equation}
Now we  split the set $\Omega_{\lambda_0 + \tau}
\setminus K$ as the union of two disjoint subsets $A_{\lambda_0 +
\tau}$ and $B_{\lambda_0 + \tau}$ such that $\Omega_{\lambda_0 +
\tau} \setminus K= A_{\lambda_0 + \tau} \cup B_{\lambda_0 + \tau}$.
In particular, for $ \dot C>1$,   we set
\begin{equation}\nonumber
\begin{split}
A_{\lambda_0 + \tau} &= \{ x \in \Omega_{\lambda_0 + \tau} \setminus
K \ : \ |
\nabla u_{\lambda_0 + \tau} (x)| < \dot C| \nabla u (x)|  \}\\ \\
\text{and}\\ \\
B_{\lambda_0 + \tau} &= \{ x \in \Omega_{\lambda_0 + \tau} \setminus
K \ : \ | \nabla u_{\lambda_0 + \tau} (x)| \geq  \dot C| \nabla u
(x)|\}.
\end{split}
\end{equation}
From \eqref{eq:Step1}, using \eqref{eq:inequalities} and {\bf
($A_f^1$)}, repeating verbatim arguments in
\eqref{eq:cambio5},
\eqref{eq:asjdhjaskhdjbjhbhjwefohi} and in
\eqref{eq:asdadajdhjaskhdjbjhbhjwefohi} we have

\begin{equation}\nonumber
\begin{split}
& \int_{\Omega_{\lambda_0 + \tau} \setminus K} (|\nabla u| + |\nabla
u_{\lambda_0 + \tau}|)^{p-2} |\nabla w_{\lambda_0 + \tau}^+|^2
\psi_\varepsilon^p \, dx
\\& \leq \delta C \int_{\Omega_{\lambda_0 + \tau} \setminus K} (|\nabla u| + |\nabla
u_{\lambda_0 + \tau}|)^{p-2} |\nabla w_{\lambda_0 + \tau}^+|^2
\psi_\varepsilon^p \, dx + C
\left(\int_{\Omega_\lambda}|\nabla u|^p\,dx \right)^{\frac{p-1}{p}}\left(
 \int_{\Omega_\lambda} |\nabla \psi_\varepsilon|^p \, dx \right)^{\frac{1}{p}} \\
&+\frac{C}{ \delta} \int_{\Omega_{\lambda_0 + \tau} \setminus K}
|\nabla \psi_\varepsilon|^p \, dx + K_f\int_{\Omega_{\lambda_0 +
\tau} \setminus K} (w_{\lambda_0+\tau}^+)^2 \psi_\varepsilon^p \, dx,
\end{split}
\end{equation}
for some positive constant $C=C(p,
\lambda,\|u\|_{L^{\infty}(\Omega_{\lambda+\bar \tau})})$.
Taking $\delta>0$ sufficiently small and using {\bf ($A_\Gamma^1$)}, as we did above passing to the
limit for $\varepsilon \rightarrow 0$ we obtain
\begin{equation}\label{eq:final}
\int_{\Omega_{\lambda_0 + \tau} \setminus K} (|\nabla u| + |\nabla
u_{\lambda_0 + \tau}|)^{p-2} |\nabla w_{\lambda_0 + \tau}^+|^2 \, dx
\leq CK_f \int_{\Omega_{\lambda_0 + \tau} \setminus K} (w_{\lambda_0
+ \tau}^+)^2 \, dx.
\end{equation}
Now we set
$$\varrho :=\left(1+|\nabla u|^2 + |\nabla u_\lambda|^2 \right)^{\frac{p-2}{2}}$$
in order to exploit the  weighted Sobolev inequality from \cite{Tru}. The results of \cite{Tru} apply if
$\varrho \in L^1(\Omega_\lambda)$ and
$$\displaystyle \frac{1}{\varrho}\in L^t(\Omega_\lambda),$$ for some $t
>{n}/{2}$. In particular, $H^1_{0,\varrho} (\Omega')$ (see \cite{DSJDE,Tru})  coincides with the closure of $C^\infty_c(\Omega')$ with respect to the norm
\[
\|w\|_\varrho\,:=\,\||\nabla w|\|_{L^2(\Omega',\varrho)}:=\left(\int_{\Omega'}\rho |\nabla w|^2\, dx \right)^{\frac 12}\,
\]
and it holds that
\[
\|w\|_{L^{2^*_\varrho}(\Omega')}\,\leq C_S \,\||\nabla w|\|_{L^2(\Omega',\varrho)}\qquad\text{for any}\,\,\,w\in H^1_{0,\varrho} (\Omega')\,,
\]
where
$$\frac{1}{2^*_\varrho}:=\frac{1}{2}\left(1+\frac{1}{t}\right)-\frac{1}{n}.$$
Note that
\begin{equation}\label{eq:weight4}
\left( 1+|\nabla u|^2+|\nabla u_{\lambda_0+\tau}|^2
\right)^{\frac{2-p}{2}} \leq K_1 + K_2 |\nabla
u_{\lambda_0+\tau}|^{2-p},
\end{equation}
in $\Omega_{\lambda_0+\tau}^+:=\Omega_{\lambda_0+\tau} \cap
\text{supp}\,(w_{\lambda_0+\tau}^+)$, where $K_1$ and $K_2$ are
positive constants depending only on $p$ and on $\|u\|_{C^1(\bar\Omega_{\lambda_0+\bar \tau})}$. By Lemma \ref{sumweight} and \eqref{eq:weight4}, we deduce that
\[
\frac{1}{\varrho}:=\left( 1+|\nabla u|^2+|\nabla
u_{\lambda_0+\tau}|^2 \right)^{\frac{2-p}{2}} \in
L^t(\Omega_{\lambda_0+\tau}),\] for some $t > {n}/{2}$ and this allows us to use the above mentioned results of \cite{Tru}.
We shall use the fact that
\begin{equation}\label{eq:weight3}
\left(|\nabla u|+|\nabla u_{\lambda_0+\tau}|\right)^{2-p} \leq
2^{\frac{2-p}{2}} \left( |\nabla u|^2+|\nabla u_{\lambda_0+\tau}|^2
\right)^{\frac{2-p}{2}} \leq 2^{\frac{2-p}{2}} \left( 1+|\nabla
u|^2+|\nabla u_{\lambda_0+\tau}|^2 \right)^{\frac{2-p}{2}}.
\end{equation}
In particular, by
 \eqref{eq:weight3}, H\"older inequality and weighted
Sobolev inequality, in
\eqref{eq:final}, we obtain
\begin{equation}\label{eq:Poincare}
\begin{split}
\int_{\Omega_{\lambda_0 + \tau} \setminus K} \varrho |\nabla
w_{\lambda_0+\tau}^+|^2 \, dx &\leq 2^{\frac{2-p}{2}}
\int_{\Omega_{\lambda_0 + \tau} \setminus K} (|\nabla u| + |\nabla
u_{\lambda_0+\tau}|)^{p-2}
|\nabla w_{\lambda_0+\tau}^+|^2 \, dx \\
& \leq 2^{\frac{2-p}{2}} C K_f \int_{\Omega_{\lambda_0 + \tau}
\setminus K}
(w_{\lambda_0+\tau}^+)^2 \, dx\\
&\leq 2^{\frac{2-p}{2}} C K_f |\Omega_{\lambda_0+\tau} \setminus
K|^{\frac{1}{(\frac{2}{2^*_\varrho})'}}
\left(\int_{\Omega_{\lambda_0 + \tau} \setminus K}
(w_{\lambda_0+\tau}^+)^{2^*_\varrho} \, dx
\right)^{\frac{2}{2^*_\varrho}}\\
%& \leq 2^{\frac{2-p}{2}} K_f |\Omega_{\lambda_0 + \tau} \setminus
%K|^{\frac{1}{(\frac{2}{2^*_\rho})'}} C(N,t) \left| \left|
%\frac{1}{\rho} \right| \right|_{L^t(\Omega_{\lambda_0 + \tau}
%\setminus K)}
%\int_{\Omega_{\lambda_0 + \tau} \setminus K} \rho |\nabla w_{\lambda_0+\tau}^+|^2  dx\\
& \leq 2^{\frac{2-p}{2}} C K_f C_p(|\Omega_{\lambda_0 + \tau}
\setminus K|)\int_{\Omega_{\lambda_0 + \tau} \setminus K} \varrho
|\nabla w_{\lambda_0+\tau}^+|^2  dx,
\end{split}
\end{equation}
where $C_p(\cdot)$ tends to zero if the measure of the domain tends to zero. For $\bar{\tau}$
small and $K$ large, we may assume that
$$2^{\frac{2-p}{2}} C K_f C_p(|\Omega_{\lambda_0+\tau} \setminus K|) < \frac{1}{2}$$
so that by \eqref{eq:Poincare}, we deduce that
$$\int_{\Omega_{\lambda_0+\tau}} \varrho |\nabla
w_{\lambda_0+\tau}^+|^2 \, dx = \int_{\Omega_{\lambda_0 + \tau}
\setminus K}  \varrho |\nabla w_{\lambda_0+\tau}^+|^2 \, dx  \leq
0,$$ proving that $u \leq u_{\lambda_0+\tau}$ in $\Omega_{\lambda_0
+ \tau} \setminus R_{\lambda_0 + \tau} (\Gamma)$ for any $0 < \tau <
\bar{\tau}$ for some small $\bar{\tau}>0$. Such a contradiction
shows that
$$ \lambda_0 = 0.$$
Since the moving plane procedure can be performed in the same way
but in the opposite direction, then this proves the desired symmetry
result. The fact that the solution is increasing in the
$x_1$-direction in $\{x_1 < 0\}$ is implicit in the moving plane
procedure.

\end{proof}
\begin{proof}[Proof of Theorem \ref{main}]

Arguing verbatim as in the previous case up to \eqref{eq:moduloooooooo},
we consider
$$\varphi_\varepsilon := w^+_{\lambda_0 + \tau} \psi_\varepsilon^p \chi_{\Omega_{\lambda_0 + \tau}}$$
and by density arguments, we plug it  as test
function in \eqref{debil1} and \eqref{debil2}. Subtracting,
we get
\begin{equation} \label{eq:Step1ggg}
\begin{split} &\int_{\Omega_{\lambda_0 + \tau} \setminus K} (|\nabla u|^{p-2} \nabla u
- |\nabla u_{\lambda_0 + \tau}|^{p-2}\nabla u_{\lambda_0 + \tau},
 \nabla w^+_{\lambda_0 + \tau}) \psi_\varepsilon^p \, dx\\
&+ p \int_{\Omega_{\lambda_0 + \tau} \setminus K} (|\nabla u|^{p-2}
\nabla u - |\nabla u_{\lambda_0 + \tau}|^{p-2} \nabla u_{\lambda_0 +
\tau}, \nabla \psi_\varepsilon)
\psi_\varepsilon^{p-1} w_{\lambda_0 + \tau}^+ \, dx \\
&= \int_{\Omega_{\lambda_0 + \tau} \setminus K} (f(u)-f(u_\lambda))
w_{\lambda_0 + \tau}^+ \psi_\varepsilon^p \, dx.
\end{split}
\end{equation}
Using the split
\begin{equation}\nonumber
\begin{split}
A_{\lambda_0 + \tau} &= \{ x \in \Omega_{\lambda_0 + \tau} \setminus
K \ : \ |
\nabla u_{\lambda_0 + \tau} (x)| < \dot C| \nabla u (x)|  \},\\
B_{\lambda_0 + \tau} &= \{ x \in \Omega_{\lambda_0 + \tau} \setminus
K \ : \ | \nabla u_{\lambda_0 + \tau} (x)| \geq  \dot C| \nabla u
(x)|\},
\end{split}
\end{equation}
from \eqref{eq:Step1ggg}, using \eqref{eq:inequalities},{\bf
($A_f^2$)} and arguing as in Lemma \ref{leaiuto}, we
obtain
\begin{equation}\nonumber
\begin{split}
&\int_{\Omega_{\lambda_0 + \tau} \setminus K} (|\nabla u| + |\nabla
u_{\lambda_0 + \tau}|)^{p-2} |\nabla w_{\lambda_0 + \tau}^+|^2
\psi_\varepsilon^p \, dx \\
&\leq \delta C \int_{\Omega_{\lambda_0 + \tau} \setminus K} (|\nabla
u| + |\nabla
u_{\lambda_0 + \tau}|)^{p-2} |\nabla w_{\lambda_0 + \tau}^+|^2 \psi_\varepsilon^p \, dx \\
&+ \frac{C}{\delta} \left(\int_{\Omega_{\lambda_0 + \tau} \setminus
K} |\nabla u|^p \, dx \right)^{\frac{p-2}{p}}
\left(\int_{\Omega_{\lambda_0 + \tau} \setminus K} |\nabla
\psi_\varepsilon|^p \, dx
\right)^{\frac{2}{p}}\\
& + K_f\int_{\Omega_{\lambda_0 + \tau} \setminus K}
(w_{\lambda_0+\tau}^+)^2 \psi_\varepsilon^p \, dx,
\end{split}
\end{equation}
for some positive constant $C=C(p,
\lambda,\|u\|_{L^{\infty}(\Omega_\lambda+\bar \tau)})$.
As we did above passing to the limit for $\varepsilon \rightarrow
0$, by Fatou's Lemma we obtain
\begin{equation}\label{eq:final2}
\int_{\Omega_{\lambda_0 + \tau} \setminus K} (|\nabla u| + |\nabla
u_{\lambda_0 + \tau}|)^{p-2} |\nabla w_{\lambda_0 + \tau}^+|^2 \, dx
\leq C K_f \int_{\Omega_{\lambda_0 + \tau} \setminus K}
(w_{\lambda_0 + \tau}^+)^2 \, dx.
\end{equation}

\noindent In this case we have $|\nabla u|^{p-2} \leq (|\nabla u| +
|\nabla u_{\lambda_0+\tau}|)^{p-2}$ since $p>2$. Then we set $\rho:=|\nabla
u|^{p-2}$ and we see that $\rho$ is bounded in
$\Omega_{\lambda_0+\tau}$, hence $\rho \in L^1(\Omega_{\lambda_0 +
\tau})$. By applying the weighted Poincar\'{e} inequality to
\eqref{eq:final2}, see \cite[Theorem $1.2$]{DSJDE}, we deduce that
\begin{equation}\label{eq:final3}
\begin{split}
\int_{\Omega_{\lambda_0 + \tau} \setminus K} \rho |\nabla
w_{\lambda_0 + \tau}^+|^2 \, dx &\leq \int_{\Omega_{\lambda_0 +
\tau} \setminus K} (|\nabla u| + |\nabla u_{\lambda_0 +
\tau}|)^{p-2} |\nabla w_{\lambda_0 + \tau}^+|^2 \, dx\\
& \leq C K_f \int_{\Omega_{\lambda_0
+ \tau} \setminus K} (w_{\lambda_0 + \tau}^+)^2 \, dx\\
&\leq C K_f C_p(|\Omega_{\lambda_0 + \tau} \setminus K|)
\int_{\Omega_{\lambda_0 + \tau} \setminus K} \rho |\nabla
w_{\lambda_0 + \tau}^+|^2 \, dx
\end{split}
\end{equation}
where $C_p(\cdot)$ tends to zero if the measure of the domain tends to zero. For $\bar{\tau}$
small and $K$ large, we may assume that
$$C K_f C_p(|\Omega_{\lambda_0 + \tau} \setminus K|) < \frac{1}{2}$$
so that by \eqref{eq:final3}, we deduce that
$$\int_{\Omega_{\lambda_0+\tau}} \rho |\nabla
w_{\lambda_0+\tau}^+|^2 \, dx = \int_{\Omega_{\lambda_0 + \tau}
\setminus K}  \rho |\nabla w_{\lambda_0+\tau}^+|^2 \, dx  \leq 0,$$
proving that $u \leq u_{\lambda_0+\tau}$ in $\Omega_{\lambda_0 +
\tau} \setminus R_{\lambda_0 + \tau} (\Gamma)$ for any $0 < \tau <
\bar{\tau}$ for some small $\bar{\tau}>0$. Such a contradiction
shows that
$$ \lambda_0 = 0.$$
Since the moving plane procedure can be performed in the same way
but in the opposite direction, then this proves the desired symmetry
result. The fact that the solution is increasing in the
$x_1$-direction in $\{x_1 < 0\}$ is implicit in the moving plane
procedure.

\end{proof}


\begin{thebibliography}{99}
\bibitem{A} {\sc A.D.~Alexandrov}, A characteristic property of the spheres.
{\em Ann. Mat. Pura Appl.} 58, 1962, pp.  303 -- 354.

\bibitem{BN} {\sc H.~Berestycki and L.~Nirenberg},
On the method of moving planes and the sliding method. {\em Bulletin
Soc. Brasil. de Mat Nova Ser}, 22(1), 1991, pp. 1--37.

\bibitem{lucio} {\sc L. Damascelli},
\newblock \emph{Comparison theorems for some quasilinear degenerate elliptic operators and applications to symmetry and monotonicity results}.
\newblock  Ann. Inst. H. Poincar\'{e}Anal. Non Lin\'{e}aire, 15(4), 1998, 493--516.


\bibitem{DamPac} {\sc L. Damascelli and F. Pacella}, Monotonicity and symmetry of solutions of p-Laplace equations, $1<p<2$,via the moving plane
method. {\em Ann. Scuola Norm. Sup. Pisa Cl. Sci.} (4) 26 (1998),
no. 4, pp. 689–-707.

\bibitem{DamSciunzi} { \sc L. Damascelli and B. Sciunzi}, Harnack inequalities, maximum and comparison principles,
and regularity of positive solutions of $m$-Laplace equations. {\em
Calc. Var. Partial Differential Equations}, 25 (2), 2006,  pp.
139--159.

\bibitem{DSJDE} { \sc L. Damascelli and B. Sciunzi}, Regularity, monotonicity and symmetry
of positive solutions of $m$-Laplace equations. {\em J. of
Differential Equations}, 206, 2004, no.2, pp. 483--515.

\bibitem{DB} {\sc E. Di Benedetto}, $C^{1+\alpha}$ local
regularity of weak solutions of degenerate elliptic equations. {\em
Nonlinear Anal.} 7(8),  1983, pp. 827--850.

\bibitem{EFS} {\sc F. Esposito, A. Farina and B. Sciunzi}, Qualitative properties of singular solutions to semilinear elliptic problems. {\em J. of Differential Equations.}
265(5), 2018,  pp. 1962--1983.


\bibitem{GNN}{\sc B.~Gidas, W.~M.~Ni and L.~Nirenberg},
Symmetry and related properties via the maximum principle. {\em
Comm. Math. Phys.} , 68, 1979, pp. 209--243.

%\bibitem{GT} {\sc D.Gilbarg and N. Trudinger},
%Elliptic partial differential equations of second order. Reprint of
%the 1998 Edition, {\em Springer}.
%
\bibitem{MMPS} {\sc S. Merch\'an, L. Montoro, I. Peral, B. Sciunzi,}
\newblock Existence and qualitative properties of solutions to  a quasilinear elliptic equation involving the Hardy-Leray potential.
\newblock {\em Ann. Inst. H. Poincar\'{e} Anal. Non Lin\'{e}aire}, 31(1), pp. 1--22, 2014.
%
%
%\bibitem{MurStam} {\sc M.K.V. Murthy and G. Stampacchia}, Boundary value problems for some
%degenerate-elliptic operators. {\em Ann. Mat. Pura Appl}. 80 (4),
%1968, pp. 1–-122.

\bibitem{PPS} {\sc F. Pacard, F. Pacella and B. Sciunzi}, Solutions of semilinear elliptic equations in tubes.
{\em J. Geom. Anal.}, 24(1), 2014, pp. 445-–471.

\bibitem{PQS} {\sc P. Pol\'acik, P. Quittner and P. Souplet}, Singularity and decay
estimates in superlinear problems via Liouville-type theorems. I.
Elliptic equations and systems. {\em Duke Math. J.} 139 (2007), no.
3, 555–-579.

\bibitem{PucciSerrin} {\sc P. Pucci and J. Serrin}
 The maximum principle. {\em Progress in Nonlinear Differential
Equations and their Applications, 73. Birkh\"{a}user Verlag}, Basel,
2007.

\bibitem{serrin} {\sc J. Serrin},
A symmetry problem in potential theory. {\em Arch. Rational Mech.
Anal.}, 43, 1971, pp. 304--318.

\bibitem{io} {\sc B. Sciunzi}, On the moving plane method for singular solutions to semilinear elliptic equations. {\em J. Math. Pure Appl.}
108(1), 2017,  pp. 111--123.


\bibitem{T} {\sc P. Tolksdorf}, Regularity for a more general class
of quasilinear elliptic equations. {\em J. Differential Equations.}
51(1), 1984,  pp. 126--150.

\bibitem{Tru} {\sc N.S. Trudinger}, Linear elliptic operators
with measurable coefficients. {\em Ann. Scuola Norm. Sup. Pisa} (3),
27, 1973, pp. 265--308.

\end{thebibliography}
\end{document}